\documentclass[12pt]{article}
\title{One-dimensional subgroups and connected components in non-abelian $p$-adic definable groups}
\author{Will Johnson and Ningyuan Yao}

\usepackage{amsmath, amssymb, amsthm}    	
\usepackage{fullpage} 	
\usepackage{amscd}
\usepackage{hyperref}
\usepackage[all]{xy}
\usepackage{centernot}

\DeclareMathOperator*{\forkindep}{\raise0.2ex\hbox{\ooalign{\hidewidth$\vert$\hidewidth\cr\raise-0.9ex\hbox{$\smile$}}}}

\newcommand{\alg}{\mathrm{alg}}

\newcommand{\Ad}{\operatorname{Ad}}
\newcommand{\ad}{\operatorname{ad}}

\newcommand{\im}{\operatorname{im}}
\newcommand{\Lie}{\operatorname{Lie}}

\newcommand{\GL}{\operatorname{GL}}

\newcommand{\Gal}{\operatorname{Gal}}

\newcommand{\Aut}{\operatorname{Aut}}

\newcommand{\Th}{\operatorname{Th}}
\newcommand{\id}{\operatorname{id}}

\newcommand{\acl}{\operatorname{acl}}

\newcommand{\img}{\operatorname{im}}

\newcommand{\Inn}{\operatorname{Inn}}

\newcommand{\Qp}{\mathbb{Q}_p}
\newcommand{\pCF}{p\text{CF}}

\newtheorem{theorem}{Theorem}[section] 
\newtheorem{lemma}[theorem]{Lemma}

\newtheorem{corollary}[theorem]{Corollary}
\newtheorem{fact}[theorem]{Fact}

\newtheorem{proposition-eh}[theorem]{Proposition(?)}
\newtheorem*{theorem-star}{Theorem}
\newtheorem*{conjecture-star}{Conjecture}
\newtheorem*{lemma-star}{Lemma}
\newtheorem*{claim-star}{Claim}

\theoremstyle{definition}
\newtheorem{definition}[theorem]{Definition}

\newtheorem{remark}[theorem]{Remark}
\newtheorem{claim}[theorem]{Claim}

\newtheorem*{acknowledgment}{Acknowledgments}

\newtheorem*{warning}{Warning}

\newcommand{\Qq}{\mathbb{Q}}

\newcommand{\an}{\textrm{an}}

\newcommand{\Rr}{\mathbb{R}}

\newcommand{\Zz}{\mathbb{Z}}
\newcommand{\Kk}{\mathbb{K}}

\newcommand{\Gg}{\mathbb{G}}

\newcommand{\Mm}{\mathbb{M}}

\newcommand{\Oo}{\mathcal{O}}

\newcommand{\sq}{\subseteq}

\newenvironment{claimproof}[1][\proofname]
               {
                 \proof[#1]
                 
               }
               {
                 \endproof
               }

\begin{document}
\maketitle

\begin{abstract}
  We generalize two of our previous results on abelian definable
  groups in $p$-adically closed fields
  \cite{J-Y-Non-compact,jy-abelian} to the non-abelian case.  First,
  we show that if $G$ is a definable group that is not definably
  compact, then $G$ has a one-dimensional definable subgroup which is
  not definably compact.  This is a $p$-adic analogue of the
  Peterzil-Steinhorn theorem for o-minimal theories
  \cite{Peterzil-Steinhorn}.  Second, we show that if $G$ is a group
  definable over the standard model $\Qq_p$, then $G^0 = G^{00}$.  As
  an application, definably amenable groups over $\Qq_p$ are open
  subgroups of algebraic groups, up to finite factors.  We also prove
  that $G^0 = G^{00}$ when $G$ is a definable subgroup of a linear
  algebraic group, over any model.
\end{abstract}

\section{Introduction}

This paper continues our earlier work
\cite{J-Y-Non-compact,jy-abelian} on definable groups in the theory
$\pCF$ of $p$-adically closed fields.  We prove two main results.  The
first is as follows: 
\begin{theorem} \label{intro1}
  Let $G$ be a definable group in a $p$-adically closed field.  If $G$
  is not definably compact, then $G$ contains a one-dimensional
  definable subgroup $H$ which is not definably compact.
\end{theorem}
See Subsections~\ref{sec:defcom}--\ref{sec:pcf} for definitions of the
relevant terms.  The analogous statement for definable groups in
o-minimal structures is the classic Peterzil-Steinhorn theorem
\cite{Peterzil-Steinhorn}.  The abelian case of Theorem~\ref{intro1}
was the main result of \cite{J-Y-Non-compact}.

Our second main result concerns the model-theoretic connected
components $G^0$ and $G^{00}$.  Recall that if $G$ is a definable
group in a monster model of an NIP theory such as $p$CF, then the collection of definable
(resp.\@ type-definable) subgroups of finite (resp.\@ bounded) index
is bounded, and the intersection is denoted $G^0$ (resp.\@ $G^{00}$)
\cite[\S8.1]{NIPguide}.  We always have $G^{00} \subseteq G^0$, and
the inclusion can be strict.  For example, if $G$ is the circle group
in RCF, then $G^0 = G$ but $G^{00}$ is an infinitesimal neighborhood
of the identity element.  Our second main theorem shows that this
does \emph{not} happen for groups definable over $\Qp$:
\begin{theorem} \label{intro2}
  Let $\Kk$ be a highly saturated elementary extension of $\Qp$ and
  let $G$ be a $\Qp$-definable group in $\Kk$.  Then $G^0 = G^{00}$.
\end{theorem}
We previously proved the abelian case in
\cite[Theorem~4.2]{jy-abelian}.  The main goal of the current note,
then, is to deal with the non-abelian cases of Theorems~\ref{intro1}
and \ref{intro2}.

In the course of proving Theorem~\ref{intro2}, we need to prove the
following variant, which is interesting in its own right:
\begin{theorem} \label{intro3}
  Let $\Kk$ be a highly saturated $p$-adically closed field, let $G$
  be a linear algebraic group over $\Kk$, and let $H \subseteq G(\Kk)$
  be a definable subgroup.  Then $H^0 = H^{00}$.
\end{theorem}
The assumption that $G$ is a \emph{linear} algebraic group is
essential; Onshuus and Pillay show that $E(\Kk)^0 \ne E(\Kk)^{00}$ for
certain elliptic curves \cite[Proposition~3.7]{O-P}.

As in \cite{jy-abelian}, Theorem~\ref{intro2} implies the following
weak classification of definably amenable groups over $\Qp$:
\begin{theorem}
  Let $G$ be a definably amenable group defined in $\Qp$.  There is a
  finite index definable subgroup $E \subseteq G$ and a finite normal
  subgroup $F \lhd E$ such that the quotient $E/F$ is isomorphic to an
  open subgroup of an algebraic group over $\Qp$.
\end{theorem}

\begin{remark}
  Tracing through the proofs of Theorems~\ref{intro2}--\ref{intro3},
  one can see that in $p$-adically closed fields, instances of $G^0
  \ne G^{00}$ must be built out of primitive instances $G$ for which
  (1) $G$ is abelian, (2) $G$ is \emph{not} $\Qq_p$-definable, and (3)
  $G$ is \emph{not} a definable subgroup of a linear algebraic group.
  These constraints might be strong enough that one could classify all
  such groups, perhaps using the techniques of \cite{acosta}.  This
  would lead to a better understanding of the structure of $G/G^{00}$
  for general definable groups $G$ in $p$CF.
\end{remark}

\begin{remark}
  The term ``$p$-adically closed field'' is often used for the more
  general class of fields elementarily equivalent to finite extensions
  $K/\Qp$.  All of our theorems generalize to this broader context.
  For simplicity we will only consider the case of $\Th(\Qp)$.  In
  most cases, the proofs generalize with minimal changes.  We leave
  the details as an exercise to the reader.
  \emph{However}, in Section~\ref{sec:last}, some of the intermediate lemmas
  fail to generalize, as explained in Remark~\ref{start-problem}.
  Nevertheless, the main theorems \emph{do} successfully generalize,
  for reasons explained in Remark~\ref{finite-extension}.
\end{remark}


\subsection{Notation and conventions}
Let $\mathcal{L}$ be a first-order language and $M$ be an $\mathcal{L}$-structure. The letters $x, y, z$ will denote finite tuples of variables, and $a, b, c$ will denote finite tuples from $M$.  For a subset $A$ of $M$, $\mathcal{L}_A$ is the language obtained from $\mathcal{L}$ by adjoining constants for elements of $A$. For an $\mathcal{L}_M$-formula $\phi(x)$, $\phi(M)$ denotes the definable subset of $M^{|x|}$ defined by $\phi$. A set $X$ is definable in $M$ if there is an $\mathcal{L}_M$-formula $\phi(x)$ such that $X = \phi(M)$.
If $M\prec N$, and $X\sq M^n$ is defined by a formula $\psi$ with parameters from $M$, then
  $X(N)$ will denote the definable set $\psi(N)$.  We will distinguish between definable and interpretable in the current paper.

\subsection{Outline} 
In Section~\ref{sec:prelim}, we review the notions of definable
compactness, $p$-adic definable groups, and $p$-adic algebraic
groups, as well as some useful tools. In Section~\ref{sec:old2}, we prove Theorem~\ref{intro1}, the
$p$-adic Peterzil-Steinhorn theorem.  In Section~\ref{chg-sec} we
collect some useful information on compact Hausdorff groups and apply it to the groups $G/G^{00}$.
Finally, in Section~\ref{sec:last} we use this machinery to prove
Theorems~\ref{intro2} and \ref{intro3}.

\section{Preliminaries} \label{sec:prelim}

\subsection{Definable Compactness} \label{sec:defcom}

We recall some notions from \cite{wj-o-minimal}.  Let $M$ be an arbitrary structure. A {\em definable topology} on a definable set $X  \sq M^n$ is a
topology with a (uniformly) definable basis of opens. A {\em definable topological space} is a
definable set with a definable topology. A definable topological space $X$ is {\em definably compact} if for any definable family ${\cal F}=\{Y_t \mid t\in T\}$ of non-empty closed sets $Y_t\sq  X$, if ${\cal F}$ is downwards directed, then $\bigcap {\cal F}\neq \emptyset$.  A definable subset $D \subseteq X$ is \emph{definably compact} if it is definably compact as a subspace.

\begin{fact}[{\cite[Section~3.1]{wj-o-minimal}}]
\begin{enumerate}
\item If $X$ is a compact definable topological space, then $X$ is definably compact.
\item If $f :X\to Y$ is definable and continuous, and $X$ is definably compact, then the
image $f(X)\sq Y$ is definably compact.
\item If $X$ is definably compact and $Y\sq X$ is closed and definable, then $Y$ is definably
compact.
\item If $X$ is Hausdorff and $Y \subseteq X$ is definable and definably compact, then $Y$ is closed.
\item If $X_1, X_2$ are definably compact spaces, then $X_1 \times X_2$ is definably compact.
\item If $X$ is a definable topological space and $Y_1, Y_2\sq X$ are definably compact, then
$Y_1\cup Y_2$ is definably compact.
\end{enumerate}
\end{fact}
\begin{remark}\label{rmk-definably-compact}
Suppose $X$ is a definable topological space in a structure $M$, and $N \succ M$.
Then $X(N)$ is naturally a definable topological space in the structure $N$, and $X(N)$
is definably compact if and only if $X$ is definably compact. In other words, definable
compactness is invariant in elementary extensions.
\end{remark}

\subsection{$\pCF$ and definable groups} \label{sec:pcf}
 Let $p$ be a prime and $\Qp$ the field of $p$-adic numbers. We call the complete theory of $\Qp$, in the language of rings, the theory of \emph{$p$-adically closed fields}, written $\pCF$. For any $K\models \pCF$,  $\Oo(K)$ will denote the valuation ring and
$\Gamma_K$ will denote the value group, which is an elementary extension of $(\mathbb{Z},+,<)$. Let $v: K\to\Gamma_K\cup\{\infty\}$ be the valuation map and
\[
B(a, \alpha) = \{x \in \Qp \mid  v(x-a)\geq \alpha\}
\]
for $a\in K$ and $\alpha \in\Gamma_K\cup\{\infty\}$.  Then $K$ is topological field with basis given by the sets $B(a, \alpha)$. The $p$-adic field $\Qp$ is locally compact.  We call $X\sq K$ \emph{bounded} if there is $\alpha\in\Gamma_K$ such that $X$ is a subset of some $n$-dimensional ball $B(0, \alpha)^n$. 
\begin{fact}[{\cite[Lemmas~2.4--2.5]{J-Y-Non-compact}}]
    Let  $X$  be a definable subset of $K^n$.  Then $X$ is definably compact iff $X$ is closed and bounded.
\end{fact}
 An $n$-dimensional \emph{definable $C^k$-manifold} over $K$ is a Hausdorff definable topological
space $X$  with a finite covering by open sets each homeomorphic to an open definable subset of $K^{n}$ with transition maps definable and $C^{k}$. 


By a \emph{definable group} over $K$, we mean a definable set with a definable group operation.  By adapting the methods of \cite{Pillay-G-in-p}
one sees that for any group $G$ definable in $K$ and for any $k<\omega$, $G$ can be definably equipped with the structure of a definable $C^{k}$-manifold in $K$ with respect to which the group structure is $C^{k}$.  Moreover, this $C^k$-manifold structure is unique.
We will always use this manifold structure when making topological statements about $G$.  For example, $G$ is ``definably compact'' if it is definably compact with respect to this $C^k$-manifold structure.

As observed in Proposition 2.1 of \cite{Hru-Pil}, $p$-adically closed fields are \emph{geometric fields}, in the sense that (1) they have uniform finiteness and (2) model-theoretic algebraic closure agrees with field-theoretic algebraic closure:
\begin{equation*}
  a \in \acl(F) \iff a \in F^{\alg}.
\end{equation*}
Consequently, there is a sensible dimension theory for definable sets.
Assuming $X\sq K^n$ is definable over a set $A$,  then the dimension $\dim(X)$ can be described as the maximum of  $\dim(\bar a/A)$  as $\bar a$ ranges over points of $X(N)$, where $N$ is an $|A|^+$-saturated elementary extension of $K$. The dimension $\dim(X)$ coincides with the algebro-geometric dimension of the Zariski closure of $X$.

\subsection{Algebraic Groups}

Let $K$ be a $p$-adically closed field and $L$ be an algebraically closed field containing $K$. Let $G$ be an algebraic group definable in $L$, with parameters from $K$,  which means that the variety structure as well as the group structure are given by data over $K$.  (See \cite{Pillay-ACF} for more details.) Then $G$ and its group operation are defined by quantifier-free formulas over $K$ in the language of rings. The $K$-points $G(K)$ of the algebraic group $G$ is of course a definable group in $K$. By a ``definable subgroup of an algebraic group,'' we mean a definable subgroup of $G(K)$ for some algebraic group $G$.


\subsection{Centralizer-connected groups}
Let $(G,\cdot)$ be a definable group in a $p$-adically closed field $K$.
The folowing definition is standard:
\begin{definition}
  $G$ is \emph{centralizer-connected} if there is no $a \in G$ such
  that the centralizer $Z_G(a)$ is a proper subgroup of $G$ of finite
  index.
\end{definition}
The proofs of the next two theorems are variants of the proof of
\cite[Proposition~2.3]{P-Y-commu-by-fin}.
\begin{theorem} \label{cc-g0}
  Let $G'$ be the intersection of all finite-index centralizers in
  $G$.  Then $G'$ is a definable subgroup of finite index in $G$.
  Moreover, $G'$ is centralizer-connected.
\end{theorem}
\begin{proof}
  Recall that $G$ acts on itself via conjugation, $Z_G(a)$ is the
  stabilizer of $a$, and the orbit of $a$ is the conjugacy class
  $a^G$.  Thus the index of $Z_G(a)$ in $G$ is the size of $a^G$, and
  $Z_G(a)$ has finite index if and only if $a^G$ is finite.  The
  theory of $p$-adically closed fields has uniform finiteness, so
  there is some $n$ such that
  \begin{equation*}
    |a^G| < \infty \implies |a^G| \le n
  \end{equation*}
  for every $a \in G$.  Equivalently,
  \begin{equation*}
    |G : Z_G(a)| < \infty \implies |G : Z_G(a)| \le n.
  \end{equation*}
  In an NIP theory such as $p$CF, if $G$ is a definable group and $\phi(x,y)$ is a formula and $n$ is an integer, then the family
  \begin{align*}
    \{H :~ & H \text{ is a subgroup of } G,\\
    & H \text{ is defined by } \phi(x,b) \text{ for some } b, \\
    & \text{and } |G : H| \le n\}
  \end{align*}
  is finite.\footnote{The intersection of this family has finite index
  by the Baldwin-Saxl theorem for NIP theories \cite[Theorem~8.3 and
    the following discussion]{NIPguide}.  Therefore the family is
  finite.}  Consequently, once we have a uniform bound $n$ on the
  index of finite-index centralizers, it follows that there are only
  finitely many finite-index centralizers.  Thus, the group $G' =
  \bigcap \{Z_G(a) : a \in G, ~ |G : Z_G(a)| < \infty\}$ is definable
  and has finite index.

  Next, suppose that $G'$ fails to be centralizer-connected, witnessed
  by some element $a \in G'$ such that
  \begin{equation*}
    1 < |G' : Z_{G'}(a)| < \infty.
  \end{equation*}
  Then $Z_{G'}(a)$ has finite index in $G'$, which has finite index in
  $G$.  As $Z_G(a)$ contains $Z_{G'}(a)$, we see that $Z_G(a)$ also
  has finite index in $G$.  Then $G' \subseteq Z_G(a)$ by choice of
  $G'$, implying that every element of $G'$ commutes with $a$.  This
  makes $Z_{G'}(a) = G'$.
\end{proof}
\begin{theorem} \label{cc2}
  Suppose $(G,\cdot)$ is centralizer-connected and non-abelian.
  \begin{enumerate}
  \item $\dim(Z_G(a)) < \dim(G)$ for every $a \in G \setminus Z(G)$.
  \item $\dim(Z(G)) < \dim(G)$.
  \end{enumerate}
\end{theorem}
\begin{proof}
  \begin{enumerate}
  \item There is an interpretable bijection between the conjugacy
    class $a^G$ and the set of cosets $G/Z_G(a)$.  By dimension
    theory,
    \begin{equation*}
      \dim(G) = \dim(Z_G(a)) + \dim(a^G).
    \end{equation*}
    Suppose that $\dim(Z_G(a)) = \dim(G)$ for the sake of
    contradiction. Then $\dim(a^G) = 0$, implying that $a^G$ is finite
    and $Z_G(a)$ has finite index in $G$.  As $G$ is
    centralizer-connected, $Z_G(a) = G$.  But then $a \in Z(G)$.
  \item Take any $a \in G \setminus Z(G)$.  Then $Z(G) \subseteq
    Z_G(a)$, so $\dim(Z(G)) \le \dim(Z_G(a)) < \dim(G)$. \qedhere
  \end{enumerate}
\end{proof}
We will also need the following related facts from \cite{P-Y-commu-by-fin}:
\begin{fact}\label{commu-open-neibor}
  Let $G$ be a group definable in $K$.  If $G$ has a commutative open neighborhood of the identity, then $G$ is commutative-by-finite.
\end{fact}
\begin{fact} \label{one-dim}
  Let $G$ be a group definable in $K$.  If $\dim(G) = 1$, then $G$ is
  commutative-by-finite.
\end{fact}

\subsection{``Affine'' groups and the adjoint action}


Let $K$ be a $p$-adically closed field.
Let $G$ be a group definable in $K$ of dimension $n$. For $g\in G$ the
map $\Inn(g): x\mapsto gxg^{-1}$ is a $C^k$ automorphism of $G$ and
thus has a differential $d(\Inn(g))_{1}$ at the
identity $1 \in G$.
The differential $d(\Inn(g))_1$ is a linear map on the tangent space $T_1 G \to T_1 G$.
The map $\Ad: g \mapsto d(\Inn(g))_{\id_G}$ is a definable group homomorphism from $G$ to $\GL(T_1 G)$, called the \emph{adjoint representation} of $G$.
\begin{theorem} \label{trivial-adjoint}
  Suppose $\Ad : G \to \GL(T_1 G)$ is trivial.
  \begin{enumerate}
  \item There is a commutative open neighborhood $U$ of $1 \in G$.
  \item $G$ is commutative-by-finite.
  \end{enumerate}
\end{theorem}
\begin{proof}
  By Fact~\ref{commu-open-neibor}, it suffices to prove part (1).  The
  statement of (1) can be expressed via infinitely-many first-order
  sentences, so we may assume that $K = \Qp$.  Then $G$ is a $p$-adic
  Lie group.  By \cite[Corollary~18.18]{schneider}, it suffices to
  show that the Lie algebra of $G$ is abelian, i.e.,
  trivial.\footnote{If $\Lie(G)$ is trivial, then $\Lie(G) \cong
  \Lie(\Qp^n)$, so \cite[Corollary~18.18]{schneider} gives isomorphic
  open subgroups of $U_1 \subseteq G$ and $U_2 \subseteq \Qp^n$.  The
  isomorphism $U_1 \cong U_2$ shows that $U_1$ is abelian.}  That is,
  we must show that $[s,t] = 0$ for $s, t \in \Lie(G)$.

  The correct way to see this is to apply the functor $\Lie(-)$ from
  Lie groups to Lie algebras to the morphism $\Ad : G \to \GL(T_1 G)$.
  The result is known to be the adjoint representation
  \begin{gather*}
    \ad : \Lie(G) \to \mathfrak{gl}(\Lie(G)) \\
    \ad(s) = [s,-]
  \end{gather*}
  (though we had trouble finding a reference for this fact in the $p$-adic Lie group
  setting).  The triviality of $\Ad : G \to \GL(T_1 G)$ implies
  triviality of $\ad(-)$, which means that $[s,t] = 0$ for any $s$ and
  $t$.

  Here is a different proof.  The fact that $G$ acts trivially on the
  tangent space implies that any vector $s \in T_1G$ extends
  uniquely to a vector field $\xi_s$ on $G$ that is both left and
  right invariant.  Indeed, if $\lambda_g$ and $\rho_g$ denote left
  and right multiplication by $g \in G$, then $\lambda_g^{-1} \circ
  \rho_g$ fixes $s$ by triviality of the adjoint representation
  $\Ad(-)$, and so $\lambda_g(s) =\rho_g(s)$.

  The Lie algebra structure on $T_1G$ is induced by the Lie algebra
  structure on right-invariant vector fields:
  \begin{equation*}
    \xi_{[s,t]} = [\xi_s,\xi_t].
  \end{equation*}
  (See \cite[p.\@ 100, Definition]{schneider}.)  However, a
  left-invariant vector field $\xi_1$ commutes with a right-invariant
  vector field $\xi_2$, by an easy calculation related to the fact
  that the action of $G$ on the left commutes with the action of $G$
  on the right.  Since $\xi_s$ and $\xi_t$ are both left-invariant and
  right-invariant, they commute.  Therefore, the Lie algebra of $G$
  is abelian.
\end{proof}
Recall that an algebraic group $G$ is said to be ``linear'' if it is
an algebraic subgroup of $\GL_n$ for some $n$.  Analogously,
\begin{definition} \label{def-aff}
  A definable group $G$ is \emph{affine} if $G$ is a definable
  subgroup of $\GL_n(K)$ for some $n$.
\end{definition}
Perhaps ``linear'' would have been a better term than ``affine'', but
it seemed helpful to use separate terminology for the two
concepts---one is a property of algebraic groups and one is a property
of definable groups.  At any rate, the two concepts are related as
follows:
\begin{enumerate}
\item If $H$ is a linear algebraic group, then any definable subgroup
  $G \subseteq H(K)$ is an affine definable group.  In other words,
  affine definable groups are exactly the definable subgroups of
  linear algebraic groups.
\item If $G \subseteq GL_n(K)$ is an affine definable group, then the
  Zariski closure of $G$ in $\GL_n$ is a linear algebraic group.
\end{enumerate}
Note that Theorem~\ref{intro3} is a statement about affine definable
groups.
\begin{lemma}\label{abelian-affine}
If $G$ is a definable group in $K$, then there is a definable short exact sequence of $K$-definable groups
\[
1\rightarrow A\rightarrow G\stackrel{\pi}{\rightarrow} H\rightarrow 1,
\] 
where $A$ is commutative-by-finite, and $H$ is an affine group.
\end{lemma}
\begin{proof}
Consider the adjoint representation $\Ad: G\rightarrow \GL(T_1 G)$ of
$G$. Let $\pi=\Ad$, $A=\ker(\Ad)$, and $H=\img(\Ad)$.  Note that the
adjoint action of $A$ is trivial, and so $A$ is commutative-by-finite by Theorem~\ref{trivial-adjoint}.
\end{proof}

\subsection{Definable compactness in extensions and quotients}
Let $G$ be a definable group in a $p$-adically closed field $K$, and
let $H$ be a definable subgroup.  As noted above, $G$ and $H$ have the
structure of definable manifolds, making them into topological groups.
\begin{fact}[{\cite[Section~5.4]{admissible}}] \label{adm1}
  The inclusion of $H$ into $G$ is a closed embedding, and a clopen
  embedding if $\dim(H) = \dim(G)$.
\end{fact}
In particular, if $G$ is a definable group and $H$ is a definable
subgroup of the same dimension, then $H$ is clopen as a subset of $G$.
\begin{corollary} \label{index}
  If $H$ has finite index in $G$, then $H$ is definably compact if and
  only if $G$ is definably compact.
\end{corollary}
\begin{proof}
  The groups $H$ and $G$ have the same dimension, so $H$ is a clopen
  subgroup of $G$.  Then $G$ is homeomorphic to the disjoint union of
  finitely many copies of $H$.
\end{proof}
Moving beyond the finite-index case, regard the interpretable set
$G/H$ as a topological space using the quotient topology.  By the
argument of \cite[Proposition~5.1]{jy-abelian}, the continuous map $G
\to G/H$ is an open map and the quotient topology is definable---or
rather, interpretable.  Consequently, it makes sense to say that $G/H$
is or isn't definably compact.
\begin{theorem} \label{quot1}
  The group $G$ is definably compact if and only if $H$ and $G/H$ are
  definably compact.
\end{theorem}
\begin{proof}
  First suppose $G$ is definably compact.  By Fact~\ref{adm1}, $H$ is
  a closed subspace of $G$, and so $H$ is definably compact.  The
  continuous surjection $G \to G/H$ shows that $G/H$ is definably
  compact.

  Conversely, suppose that $H$ and $G/H$ are definably compact.  By
  Proposition~2.8 in \cite{J-Y-Non-compact}, there is a definable
  family of sets $\{U_\gamma\}_{\gamma \in \Gamma}$ such that (1) each
  $U_\gamma$ is open and definably compact, (2) the family is
  increasing in the sense that
  \begin{equation*}
    \gamma \le \gamma' \implies U_\gamma \subseteq U_{\gamma'},
  \end{equation*}
  and (3) $G = \bigcup_{\gamma \in \Gamma} U_\gamma$.  Let $f : G \to
  G/H$ be the quotient map.  Because $f$ is a continuous open map,
  each set $f(U_\gamma)$ is open and definably compact.  By definable
  compactness of $G/H$, there is some $\gamma$ such that $f(U_\gamma)
  = G/H$, implying that $G \subseteq U_\gamma \cdot H$.  Then $G$ is
  the image of the definably compact space $U_\gamma \times H$ under
  the continuous map $(x,y) \mapsto x \cdot y$, so $G$ is definably
  compact.
\end{proof}
A nearly identical proof shows that
\begin{theorem} \label{quot2}
  If $K = \Qq_p$, then the group $G$ is compact if and only if $H$ and
  $G/H$ are compact.
\end{theorem}
In fact, Theorem~\ref{quot2} follows from Theorem~\ref{quot1}, because
definable compactness agrees with compactness for interpretable
topological spaces in $\Qq_p$ \cite[Theorem~8.15]{AGJ}, but this is
overkill.
\begin{fact}[{\cite[Proposition~5.19]{admissible}}]
  In the case where $H$ is a normal subgroup and the quotient $G/H$ is
  definable (rather than interpretable), the quotient topology on
  $G/H$ agrees with the definable manifold topology as a definable
  group.
\end{fact}
\begin{corollary} \label{short-exact}
  Let $1 \to A \to B \to C \to 1$ be a short exact sequence of
  definable groups.  Regard $A, B, C$ as definable manifolds.
  \begin{enumerate}
  \item The maps $A \to B$ and $B \to C$ are continuous.
  \item The map $A \to B$ is a closed embedding.
  \item The map $B \to C$ is an open map.
  \item \label{se4} $B$ is definably compact if and only if $A$ and $C$ are
    definably compact.
  \end{enumerate}
\end{corollary}

\section{The $p$-adic Peterzil-Steinhorn theorem} \label{sec:old2}

In this section, we will prove the following theorem.
\begin{theorem} \label{main-thm-1}
    Let $G$ be a group definable in a $p$-adically closed field $K$. If $G$ is not definably compact, then $G$ contains a one-dimensional subgroup which is not definably compact.
\end{theorem}
Let $G$ be a group definable in $K$. Say that $G$ is \emph{nearly
abelian} if there is a definably compact definable normal subgroup
$O\subseteq G$ with $G/O$ abelian.
The following was proved in \cite{J-Y-Non-compact}:
\begin{fact}\label{abelian-case}
 If $G$ is not definably compact and $G$ is nearly abelian, then there is a one-dimensional definable
subgroup $H\sq G$ that is not definably compact.
\end{fact}

\subsection{Reduction to the standard model $\Qq_p$} \label{reduction}
We first show that Theorem~\ref{main-thm-1} is independent of $K$, by
finding an equivalent condition which depends only on $\Th(K)$.
\begin{definition} \label{special-counter}
  Let $(G,\cdot)$ be a definable group in a $p$-adically closed field
  $K$.
  \begin{enumerate}
  \item $(G,\cdot)$ is a \emph{counterexample} if $G$ is not definably
    compact, but every one-dimensional definable subgroup of $G$ is
    definably compact.  In other words, $G$ is a counterexample to
    Theorem~\ref{main-thm-1}.
  \item $(G,\cdot)$ is a \emph{special counterexample} if $G$ is not
    definably compact, but the center $Z(G)$ is definably compact, and
    the centralizer $Z_G(a)$ is definably compact for any $a \in G
    \setminus Z(G)$.
  \end{enumerate}
\end{definition}
\begin{lemma} \label{unspecial}
  If $G$ is a special counterexample, then $G$ is a counterexample.
\end{lemma}
\begin{proof}
  Otherwise, there is a 1-dimensional definably non-compact subgroup
  $H \subseteq G$.  By Fact~\ref{one-dim}, there is a finite-index
  abelian definable subgroup $H' \subseteq H$.  Then $H'$ is not
  definably compact (Corollary~\ref{index}).  Replacing $H$ with $H'$,
  we may assume that $H$ is abelian.  By Corollary~\ref{short-exact}(\ref{se4}),
  $H$ cannot be contained in any definably compact definable subgroups
  of $G$.  In particular, $H \not \subseteq Z(G)$.  Take $a \in H
  \setminus Z(G)$.  Then $H \not \subseteq Z_G(a)$, which means that
  $H$ is non-abelian, a contradiction.
\end{proof}
\begin{lemma} \label{xor}
  If $G$ is a counterexample and $H$ is a definable subgroup, then $H$
  is definably compact or $H$ is a counterexample.
\end{lemma}
\begin{proof}
  Any 1-dimensional definably non-compact subgroup of $H$ would be a
  1-dimensional definably non-compact subgroup of $G$.
\end{proof}
\begin{lemma} \label{index-ct}
  Let $G$ be a definable group and $H$ be a definable subgroup of
  finite index.  Then $H$ is a counterexample if and only if $G$ is a
  counterexample.
\end{lemma}
\begin{proof}
  By Corollary~\ref{index}, $G$ is definably compact if and only if
  $H$ is.  Suppose neither group is definably compact.  If $C$ is a
  1-dimensional definably non-compact subgroup of $H$, then $C$ is
  also a 1-dimensional definably non-compact subgroup of $G$.
  Conversely, if $C$ is a 1-dimensional definably non-compact subgroup
  of $G$, then $C \cap H$ is a 1-dimensional subgroup of $H$ which is
  definably non-compact because it has finite index in $C$, using
  Corollary~\ref{index} again.
\end{proof}

\begin{theorem} \label{to-special}
  For a fixed $p$-adically closed field $K$, the following are
  equivalent:
  \begin{enumerate}
  \item There is a counterexample $G$.
  \item There is a special counterexample $G$.
  \end{enumerate}
\end{theorem}
\begin{proof}
  (2)$\implies$(1) is Lemma~\ref{unspecial}.  For (1)$\implies$(2),
  suppose there is at least one counterexample.  Take a counterexample
  $G$ minimizing $\dim(G)$.  By Theorem~\ref{cc-g0} and
  Lemma~\ref{index-ct}, we may replace $G$ with a finite index
  subgroup and assume that $G$ is centralizer-connected.  If $G$ is
  abelian, then $G$ is \emph{not} a counterexample, by
  Fact~\ref{abelian-case}.  Therefore $G$ is non-abelian, and
  Theorem~\ref{cc2} applies, showing that
  \begin{gather*}
    \dim(Z(G)) < \dim(G) \\
    \dim(Z_G(a)) < \dim(G) \text{ for } a \in G \setminus Z(G).
  \end{gather*}
  Then $Z(G)$ and $Z_G(a)$ are not counterexamples, by choice of $G$.
  By Lemma~\ref{xor}, they must be definably compact, making $G$ be a
  special counterexample.
\end{proof}
\begin{remark}
  Let $\{G_t\}_{t \in X}$ be a definable family of definable groups.
  \begin{enumerate}
  \item The set $\{t \in X : G_t \text{ is definably compact}\}$ is
    definable \cite[Theorem~6.6]{admissible}.
  \item The set $\{t \in X : G_t \text{ is a special
    counterexample}\}$ is definable.  This follows almost immediately
    from the previous point.
  \end{enumerate}
\end{remark}
Consequently, if $K \equiv \Qq_p$, then there is a special
counterexample in $K$ if and only if there is a special counterexample
in $\Qq_p$.\footnote{To see this, embed $K$ and $\Qq_p$ both into a
highly saturated monster model $\Kk$.  If $G$ is a special
counterexample in $K$, it sits inside a 0-definable family $\{G_t\}_{t
  \in X}$ of definable groups.  The set $\{t \in X : G_t \text{ is a
  special counterexample}\}$ is definable and $\Aut(\Kk)$-invariant,
hence 0-definable.  Then it must contain a $\Qq_p$-definable point by
Tarski-Vaught, and so there is a special counterexample over $\Qq_p$.
The other direction is similar.}  By Theorem~\ref{to-special}, $K$ has
a counterexample if and only if $\Qq_p$ has a counterexample.
Therefore, \emph{in Theorem~\ref{main-thm-1}, we may assume that $K =
\Qq_p$}.

\subsection{The case of $\Qq_p$}
Now assume that $K$ is $\Qq_p$.  We prove Theorem~\ref{main-thm-1} by
induction on the dimension of $G$.  The assumption that $K = \Qq_p$
will only be used in the proof of Lemma~\ref{affine-case}.

Say that a definable group $G$ is \emph{good} if it is not a
counterexample to Theorem~\ref{main-thm-1}, meaning that either $G$ is
definably compact, or $G$ has a 1-dimensional definably non-compact
subgroup.
\begin{lemma} \label{good}
  Suppose that $1 \to A \to B \to C \to 1$ is a short exact sequence
  of definable groups, and $A$ and $C$ are good.  Then $B$ is good.
\end{lemma}
\begin{proof}
  If $B$ is definably compact, there is nothing to prove.  Suppose $B$
  is definably non-compact.  If $A$ is definably non-compact, it has a
  1-dimensional definably non-compact subgroup $X$, which shows that
  $B$ is good.  Otherwise, $A$ is definably compact.  Then $C$ is
  definably non-compact by Corollary~\ref{short-exact}.  As $C$ is
  good, there is a 1-dimensional definably non-compact subgroup $X
  \subseteq C$.  By Fact~\ref{one-dim}, we may replace $X$ with a
  finite index subgroup, and assume that $X$ is abelian.  Let $X^*
  \subseteq B$ be the preimage of $X$ under $B \to C$.  The short
  exact sequence $1 \to A \to X^* \to X \to 1$ shows that $X^*$ is
  definably non-compact (by Corollary~\ref{short-exact}) and nearly
  abelian (as $X$ is abelian and $A$ is definably compact).  By the
  nearly abelian case (Fact~\ref{abelian-case}), $X^*$ has a
  1-dimensional definably non-compact subgroup, which shows that $B$
  is good.
\end{proof}

\begin{lemma} \label{solvable}
Suppose that $G$ is a definable group contained in a solvable linear
algebraic group $B(\Qp)$.  Then $G$ is good: if $G$ is non-compact, then
$G$ has a one-dimensional definable subgroup which is not definably
compact.
\end{lemma}
\begin{proof}
  Proceed by induction on the solvable length of $B$, the length of
  the derived series.  If the derived length is $\le 1$, then $B$ and
  $G$ are abelian, and $G$ is good by the abelian case
  (Fact~\ref{abelian-case}).  Otherwise, there is a normal algebraic
  subgroup $B_1 \subseteq B$ such that the algebraic groups $C :=
  B/B_1$ and $B_1$ have lower solvable length.  Let $f : G \to C(\Qp)$
  be the composition
  \begin{equation*}
    G \hookrightarrow B(\Qp) \to C(\Qp).
  \end{equation*}
  The kernel is $G \cap B_1(\Qp)$, which is good by induction.  The
  image is a definable subgroup of $C(\Qp)$, which is good by induction.
  By Lemma~\ref{good}, $G$ is good.
\end{proof}

Recall from Definition~\ref{def-aff} that a definable group is
\emph{affine} if it is a subgroup of some linear algebraic group
$G(\Qp)$.  Lemma~\ref{solvable} shows that certain affine groups are
good, and we next generalize this to all affine groups.  But first, we
need a lemma.

\begin{lemma}\label{lemma-closed-subgroup}
Let $D$, $H$ and $A$ be topological groups, with $H$ an open subgroup of
$D$, and $A$ a subgroup of $D$.
\begin{enumerate}
\item In the quotient topology on $D/A$, the subset $H/(A \cap H)$ is
  clopen.
\item The quotient topology on $H/(A \cap H)$ as a quotient of $H$
  agrees with the subspace topology as a subspace of $D/A$.
\item If the quotient topology on $D/A$ is compact, then the quotient
  topology on $H/(A \cap H)$ is compact.
\end{enumerate}
\end{lemma}
\begin{proof}
  \begin{enumerate}
  \item Note that $D$ acts on $D/A$ on the left, and the subset $H/(A
    \cap H)$ is $\{hA : h \in H\}$, which is the $H$-orbit of $1A \in
    D/A$.  It suffices to show that every $H$-orbit is open, which
    implies then that every $H$-orbit is closed.

    Let $\pi : D \to D/A$ be the quotient map $\pi(x) = xA$.  For any
    $d \in D$, the $H$-orbit of $dA \in D/A$ is $\{hdA : h \in H\}$,
    whose preimage under $\pi$ is $\{hda : h \in H,~ a \in A\} = HdA$.
    This set is open, because it is a union of right-translates of the
    open subgroup $H$.  By definition of the quotient topology, $\{hdA
    : h \in H\}$ is open in $D/A$.
  \item Let $S$ be a subset of $H/(A \cap H)$.  Let $Q = \{h \in H :
    \pi(x) \in S\}$.  Then $S$ is open in the quotient topology if and
    only if $Q$ is open in $H$ or equivalently in $D$.  As $H/(A \cap
    H)$ is open in $D/A$, $S$ is open in the subspace topology if and
    only if $S$ is open as a subset of $D/A$, meaning that $Q' = \{d
    \in D : \pi(d) \in S\}$ is open.  So we must show that $Q$ is open
    if and only if $Q'$ is open.  Note that $Q = Q' \cap H$, and $H$
    is open, so openness of $Q'$ implies openness of $Q$.
    \begin{claim}
      $Q' = Q \cdot A$.
    \end{claim}
    \begin{claimproof}
      If $h \in Q$ and $a \in A$, then $\pi(ha) = haA = hA = \pi(h)
      \in S$, so $ha \in Q'$.  Conversely, suppose $d \in Q'$.  Then
      $\pi(d) = dA \in S$.  As $S \subseteq H/(H \cap A) = \pi(H)$, we
      have $\pi(d) = \pi(h)$ for some $h \in H$.  The fact that $dA =
      hA$ means that $d = ha$ for some $a \in A$.  Also, $\pi(h) =
      \pi(d) \in S$, so $h \in Q$.  Then $d = ha \in Q \cdot A$.
    \end{claimproof}
    If $Q$ is open, then $Q'$ is open, being a union of
    right-translates of $Q$.
  \item This follows from the previous two points---if $D/A$ is
    compact, then the closed subspace $H/(H \cap A)$ is compact, and
    this space is homeomorphic to the quotient space $H/(H \cap
    A)$. \qedhere
  \end{enumerate}
\end{proof}

\begin{lemma} \label{affine-case}
  If $G$ is affine, then $G$ is good.
\end{lemma}
\begin{proof}
  Suppose $G$ is a definable subgroup of $V(\Qp)$ for some linear algebraic
  group $V$ over $\Qp$.  Replacing $V$ with the Zariski closure of
  $G$, we may assume that $G$ is Zariski dense in $V$.  Then $G$ is
  open in $V(\Qp)$, because $G$ and $V(\Qp)$ have the same dimension as definable groups.  Let $B$ be a maximal connected $K$-split solvable
  algebraic subgroup of $V$.  By Theorem~3.1 of \cite{P-R-AG-book},
  the quotient space $V(\Qp)/B(\Qp)$ is compact.  By
  Lemma~\ref{lemma-closed-subgroup}, $G/(G \cap B(\Qp))$ is compact.
  If $G$ is compact, then
  $G$ is good.  Suppose $G$ is not compact.  Theorem~\ref{quot2} shows that $G \cap
  B(\Qp)$ is non-compact.  The group $G \cap B(\Qp)$ is also definably non-compact, since compactness and definable compactness agree for definable manifolds over the standard model \cite[Remark~2.12]{J-Y-Non-compact}.
  
  On the other hand, $G \cap B(\Qp)$ is good by
  Lemma~\ref{solvable}, so it contains a 1-dimensional non-compact
  definable subgroup.  Then $G$ has a 1-dimensional non-compact
  definable subgroup, as desired.
\end{proof}
Finally, we can complete the proof of the $p$-adic Peterzil-Steinhorn
theorem:
\begin{proof}[Proof (of Theorem~\ref{main-thm-1})]
  By Subsection~\ref{reduction}, we may assume $K = \Qq_p$.  Let $G$
  be a definable group in $\Qp$.  Lemma~\ref{abelian-affine} gives a
  short exact sequence of definable groups
  \begin{equation*}
    1 \to A \to G \to H \to 1
  \end{equation*}
  where $A$ is abelian-by-finite and $H$ is affine.  The group $A$ is
  good by the abelian case (Fact~\ref{abelian-case}, together with
  Lemma~\ref{index-ct}), and the group $H$ is good by the affine case
  (Lemma~\ref{affine-case}).  Then $G$ is good by Lemma~\ref{good}.
\end{proof}

\section{Compact Hausdorff groups and $G/G^{00}$} \label{chg-sec}
In this section, we review some facts about compact Hausdorff groups,
and apply them to the groups $G/G^{00}$.
\begin{fact} \phantomsection \label{kfact}
  \begin{enumerate}
  \item \label{kf1} A compact Hausdorff group $G$ is profinite if and
    only if $G$ is totally disconnected \cite[Theorem~1.34]{h-m}.
  \item \label{kf2} If $G$ is profinite and $f : G \to H$ is a
    continuous surjection onto another compact Hausdorff group $H$,
    then $H$ is profinite \cite[Exercise~E1.13]{h-m}.
  \item \label{kf3} If $G$ is a compact Hausdorff group, then the
    family of continuous homomorphisms from $G$ to the orthogonal
    groups $\mathrm{O}(2), \mathrm{O}(3), \ldots$ separates points
    \cite[Corollary~2.28]{h-m}.
  \item \label{kf4} Any closed, compact subgroup of $\mathrm{O}(n)$ is
    a compact Lie group \cite[Corollary~2.40, Definition~2.41]{h-m}.
  \item \label{kf5} Any profinite compact Lie group is finite
    \cite[Exercise~E2.8]{h-m}.
  \item \label{kf6} If $G$ is a non-discrete compact Lie group, then
    $G$ has a non-trivial one-parameter subgroup, meaning a
    non-trivial continuous homomorphism $\Rr \to G$
    \cite[Theorem~5.41(iv)]{h-m} (see also Definition~5.7 and
    Proposition~5.33(iv) in \cite{h-m}).
  \end{enumerate}
\end{fact}
\begin{corollary} \label{cor:prof3}
  Let $1 \to A \to B \to C \to 1$ be a continuous short exact sequence
  of compact Hausdorff groups.  Then $B$ is profinite if and only if
  $A$ and $C$ are profinite.
\end{corollary}
\begin{proof}
  If $B$ is profinite, then $C$ is profinite by
  Fact~\ref{kfact}(\ref{kf2}), $B$ is totally disconnected by
  Fact~\ref{kfact}(\ref{kf1}), the subspace $A$ is totally
  disconnected, and then $A$ is profinite by
  Fact~\ref{kfact}(\ref{kf1}).

  Conversely, suppose $A$ and $C$ are profinite, or equivalently,
  totally disconnected.  Any connected component $X \subseteq B$ maps
  onto a connected set in $C$, which must be a single point.  Then $X$
  is contained in a coset of $A$, but each such coset is totally
  disconnected because $A$ is.  Therefore $X$ is a point, and $B$ is
  totally disconnected, hence profinite.
\end{proof}
\begin{corollary} \label{image}
  Let $G$ be a compact Hausdorff group.  Then $G$ is profinite if and
  only if every continuous homomorphism $f : G \to \mathrm{O}(n)$ has finite
  image.
\end{corollary}
\begin{proof}
  If $G$ is profinite and $f : G \to \mathrm{O}(n)$ is a continuous
  homomorphism, then the image is profinite by
  Fact~\ref{kfact}(\ref{kf2}), a compact Lie group by (\ref{kf4}), and
  finite by (\ref{kf5}).

  Conversely, suppose every continuous homomorphism from $G$ to an
  orthogonal group has finite image.  Let $\{f_i\}_{i \in I}$
  enumerate all continuous homomorphisms $f_i : G \to \mathrm{O}(n_i)$.  By
  assumption, $\im(f_i)$ is finite for each $i$.  Consider the product
  homomorphism
  \begin{equation*}
    \prod_{i \in I} f_i : G \to \prod_{i \in I} \im(f_i).
  \end{equation*}
  By Fact~\ref{kfact}(\ref{kf3}), this map is injective, hence an
  embedding.  Then $G$ is a closed subgroup of the profinite group
  $\prod_{i \in I} \im(f_i)$, so $G$ is itself profinite.
\end{proof}
\begin{corollary} \label{cor:non-torsion}
  Let $G$ be an infinite compact subgroup of $\mathrm{O}(n)$.  Then
  $G$ contains a non-torsion element.
\end{corollary}
\begin{proof}
  By Fact~\ref{kfact}(\ref{kf4}), $G$ is a compact Lie group.  Since
  $G$ is infinite, it is non-discrete.  By
  Fact~\ref{kfact}(\ref{kf6}), there is a non-trivial continuous
  homomorphism $f : \Rr \to G$.  For $n \ge 1$ let $S_n$ be the closed
  subgroup \[ S_n = \{t \in \Rr : f(t)^n = 1\} = \{t \in \Rr : f(nt) =
  1\} = n^{-1} \ker(f).\] If every element of $G$ is torsion, then
  $\bigcup_{n = 1}^\infty S_n = \Rr$, and so some $S_n$ has non-empty
  interior by Baire category.  But then $S_n$ is a clopen subgroup of
  $\Rr$, so $S_n = \Rr$.  As $S_n = n^{-1} \ker(f)$, this implies
  $\ker(f) = \Rr$, contradicting the non-triviality of $f$.
\end{proof}
Recall that if $G$ is a definable group in a highly saturated
structure and $G^{00}$ exists, then the quotient $G/G^{00}$ is
naturally a compact Hausdorff group with respect to the \emph{logic
topology}, the topology where a set $X \subseteq G/G^{00}$ is closed
iff its preimage in $G$ is type-definable.  Similarly, $G/G^0$ is a
compact Hausdorff group under its logic topology.  The group $G^0$ is
the intersection of the definable subgroups of finite index, and these
correspond to the clopen subgroups of $G/G^{00}$, so the group $G/G^0$
is precisely the maximal profinite quotient of $G/G^{00}$.  In
particular, $G^0 = G^{00}$ if and only if $G/G^{00}$ is already
profinite.

In light of this, we get the following corollaries of the facts above.
\begin{corollary} \label{2-to-3}
  Let $1 \to A \to B \to C \to 1$ be a short exact sequence of
  definable groups in an NIP theory.  If $A^0 = A^{00}$ and $C^0 =
  C^{00}$, then $B^0 = B^{00}$.
\end{corollary}
\begin{proof}
  See Lemma~2.2 in \cite{jy-abelian}, which assumed
  Corollary~\ref{cor:prof3} above.
\end{proof}
\begin{remark} \label{finite-g00-rem}
  Let $A \subseteq B$ be definable groups in an NIP theory, with $|B :
  A| < \infty$.  Then $A^0 = A^{00}$ iff $B^0 = B^{00}$.  Indeed, the
  fact that $B$ has finite index implies that $B^0 = A^0$ and $B^{00}
  = A^{00}$.
\end{remark}
\begin{corollary} \label{reducts}
  Let $\Mm$ be a monster model of an NIP $\mathcal{L}$-theory, and let
  $\mathcal{L}_0$ be a sublanguage.  Let $G$ be a group definable in
  the reduct $\Mm \restriction \mathcal{L}_0$.  If $G^0 = G^{00}$ in
  $\Mm$, then $G^0 = G^{00}$ in $\Mm \restriction \mathcal{L}_0$.
\end{corollary}
\begin{proof}
  Let $H_1$ and $H_2$ be $G^{00}$ in
  $\Mm$ and $G^{00}$ in $\Mm \restriction \mathcal{L}_0$,
  respectively.  In the original structure $\Mm$, the group $H_2$ is a
  type-definable subgroup of $G$ of bounded index, so $H_2 \supseteq
  H_1$.  Then $G/H_2$ is a quotient of $G/H_1$, so profiniteness of
  $G/H_1$ implies profiniteness of $G/H_2$ by
  Fact~\ref{kfact}(\ref{kf2}).
\end{proof}
\begin{corollary} \label{to-abelian}
  Let $G$ be a definable group in a monster model $\Mm$ of an NIP
  theory.  Suppose $G^0 \ne G^{00}$.  Then there is an abelian
  definable subgroup $H \subseteq G$ such that $H^0 \ne H^{00}$.
\end{corollary}
\begin{proof}
  The group $G/G^{00}$ isn't profinite so there is a continuous
  homomorphism $f : G/G^{00} \to \mathrm{O}(n)$ with infinite image
  (Corollary~\ref{image}).  By Corollary~\ref{cor:non-torsion}, there
  is some $a \in \mathrm{O}(n)$ such that $a$ isn't torsion.  Write
  $a$ as $f(g)$ for some $g \in G$.  Then $f(g^n) = a^n \ne 1$ for all
  $n$.  Let $H$ be the center of the centralizer of $g$.  Then $H$ is
  an abelian definable subgroup of $G$ and $g \in H$.  Let $f' :
  H/H^{00} \to \mathrm{O}(n)$ be the composition
  \begin{equation*}
    H/H^{00} \to G/G^{00} \stackrel{f}{\to} \mathrm{O}(n).
  \end{equation*}
  Then $f'$ is a continuous homomorphism, and $f'(g) = a$.  Again,
  $f'(g^n) = a^n$ for all $n$.  Then the image of $f'$ contains the
  infinite cyclic group $\langle a \rangle$, and so $H/H^{00}$ isn't
  profinite (Corollary~\ref{image}).
\end{proof}
\begin{warning}
  It might appear that we can now complete the proof of
  Theorem~\ref{intro2} as follows.  Suppose for the sake of
  contradiction that $G$ is a $\Qp$-definable group $G$ in a highly
  saturated elementary extension $\Kk \succeq \Qp$, and $G^0 \ne
  G^{00}$.  Applying Corollary~\ref{to-abelian} we get an abelian
  definable subgroup $H \subseteq G$ such that $H^0 \ne H^{00}$,
  contradicting the abelian case of Theorem~\ref{intro2}.  But the
  abelian case of Theorem~\ref{intro2} was proven previously
  \cite[Theorem~4.2]{jy-abelian}.

  This proof doesn't work, because the group $H$ from
  Corollary~\ref{to-abelian} might not be $\Qp$-definable, and then
  the abelian case of Theorem~\ref{intro2} won't be applicable.
\end{warning}

\section{$G^0$ vs $G^{00}$} \label{sec:last}
In this section, we verify Theorems~\ref{intro2}--\ref{intro3}
comparing $G^0$ and $G^{00}$.  Our strategy will be to first consider
the expansion of $p$CF by restricted analytic functions.  In this
expansion, there are exponential and logarithm maps for abelian
$p$-adic Lie groups, which allow us to locally identify algebraic tori
with vector groups, simplifying the problem and mostly reducing to the
case of vector groups.  After proving the main theorems in the
analytic setting, we transfer the results back to the base theory
$p$CF via Corollary~\ref{reducts}.  We learned this trick from the
work of Acosta L\'opez \cite[\S5]{acosta}.  A variant also appears in
\cite[\S3]{O-P}.

\textbf{Until Theorem~\ref{main-thm-2}, work in the following setting.}  Let
$\Qq_{p,\an}$ be the expansion of $\Qp$ by restricted analytic
functions as in \cite[Section~3]{classic-mod-th-fields}.  The theory of $\Qq_{p,\an}$ is P-minimal \cite{analytic},
hence NIP.  Let $\Kk$ be a monster model of $\Qq_{p,\an}$.
\begin{remark} \label{start-problem}
  In most of this paper, all the results trivially generalize from
  $\Qp$ to finite extensions of $\Qp$.  In this section, we really
  need to be working with $\Qp$ rather than a finite extension.
  Later, however, we will generalize from $\Qp$ to finite extensions
  (Remark~\ref{finite-extension}).
\end{remark}

\begin{lemma} \label{subs-of-K^n}
  Let $G$ be a definable subgroup of $(\Kk^n,+)$.  Then
  \begin{equation*}
    G = \bigoplus_{i = 1}^n a_i \cdot H_i,
  \end{equation*}
  where $\{a_1,\ldots,a_n\}$ is an $\Kk$-linear basis of $\Kk^n$, and
  each $H_i \in \{0,\Oo,\Kk\}$, for $\Oo$ the valuation ring.
\end{lemma}
For $n = 1$, this was proven by Acosta L\'opez \cite{acosta} (see
Proposition~4.6, plus remarks above Lemma~5.2).
\begin{proof}
  The lemma is equivalent to a conjunction of first-order sentences,
  so we may replace $\Kk$ with the standard model $\Qq_{p,\an}$.
  Definable groups are always closed\footnote{This follows from
  dimension theory: if $G$ isn't closed then the frontier $\partial G
  := \overline{G} \setminus G$ is non-empty.  If $u \in G$ then
  translation $x \mapsto u+x$ preserves $G$ so it preserves the
  frontier $\partial G$.  That is, $G + \partial G = \partial G$, and
  then $\partial G$ is a union of cosets of $G$.  But P-minimal
  structures like $\Qq_{p,\an}$ have a nice dimension theory with the
  small boundary property: $\dim(\partial G) < \dim(G)$.  (See
  \cite[Theorem~3.5]{p-minimal-cells}.)  This contradicts the fact
  that $\partial G$ contains a coset of $G$.}, so $G$ is closed as a
  subset of $\Qq_p^n$.  We claim that $G$ is a $\Zz_p$-submodule of
  $\Qq_p^n$.  It is certainly closed under addition and negation.
  Suppose $a \in \Zz_p$ and $v \in G$.  Write $a$ as $\lim_{n \to
    \infty} a_n$ with $a_n \in \Zz$.  (Here we use the fact that $\Zz$
  is dense in $\Qq_p$, which would fail in a finite extension of
  $\Qq_p$!)  Then $av = \lim_{n \to \infty} a_nv$.  Each element
  $a_nv$ is in $G$ because $G$ is a group (a $\Zz$-module), and then
  the limit $av$ is in $G$ because $G$ is closed.

  Now proceed as in the proof of \cite[Theorem~2.6]{acvf-ei-again},
  using the fact that $\Qp$ is spherically complete.
\end{proof}
\begin{theorem} \label{additive-g00}
  If $G$ is a definable subgroup of $(\Kk^n,+)$, then $G^0 = G^{00}$.
\end{theorem}
\begin{proof}
  $G$ is definably isomorphic to a direct sum of some copies of $\Kk$
  and and $\Oo$, so we reduce to showing that $\Kk^0 = \Kk^{00}$ and
  $\Oo^0 = \Oo^{00}$.

  If $\mathcal{L}$ denotes the original language of $\Qp$, and
  $\mathcal{L}_{\an}$ denotes the language of the expansion
  $\Qq_{p,\an}$, then $\Kk$ and $\Kk \restriction \mathcal{L}$ have the
  same definable sets in one variable, because $\Qq_{p,\an}$ is
  P-minimal.  Therefore, $\Kk$ and $\Kk \restriction \mathcal{L}$ also
  have the same type-definable sets and type-definable groups, and we
  can calculate the connected components $\Kk^{00}$ and $\Oo^{00}$ in
  the reduct $\Kk \restriction \mathcal{L}$.  In other words, we can
  move to the original theory $p$CF rather than the analytic
  expansion.  Then it is well-known that
  \begin{align*}
    \Kk^{00} = \Kk^0 &= \Kk \\
    \Oo^{00} = \Oo^0 &= \bigcap_{n = 1}^\infty p^n \Oo.
  \end{align*}
  Alternatively, $\Kk^{00} = \Kk^0$ and $\Oo^{00} = \Oo^0$ hold by
  Theorem~4.2 in \cite{jy-abelian}.
\end{proof}
\begin{remark}
  When $n = 1$, Lemma~\ref{subs-of-K^n} says that the only definable
  subgroups of $(\Kk,+)$ are $0$, $\Kk$, and balls $a\Oo$.  This would
  fail if we were working with the theory of some finite extension
  $K/\Qp$ rather than $\Qp$.  For example, if $K = \Qq_3(\sqrt{-1})$,
  then the subring $\Zz_3[3\sqrt{-1}] \subseteq K$ is definable, but
  not an $\Oo_K$-submodule of $(K,+)$.

  On the other hand, Theorem~\ref{additive-g00} continues to hold,
  essentially because we can interpret $(K^n,+)$ as $(\Qq_p^{dn},+)$,
  for $d = [K : \Qq_p]$.  See Remark~\ref{finite-extension} for the
  details.
\end{remark}
\begin{fact}\label{gamma-subs-fact}
  If $G$ is a definable subgroup of $(\Gamma,+)$, then $G$ is
  $n\Gamma$ for some $0 \le n < \omega$.
\end{fact}
This is \cite[Lemma~3.3]{acosta}, modulo the fact that
$\Gamma$ is a pure model of Presburger arithmetic
\cite[Theorem~6]{cluckers}.
\begin{corollary} \label{gamma-subs-2}
  If $G$ is a definable subgroup of $(\Gamma,+)$, then $G^0 = G^{00}$.
\end{corollary}
\begin{proof}
  By Fact~\ref{gamma-subs-fact}, $G$ is trivial or isomorphic to
  $(\Gamma,+)$.  In both these cases, $G^0 = G^{00}$ is known.  For
  example, one roundabout way to see that $\Gamma^{00} = \Gamma^0$ is
  to use the fact that $\Gamma/\Gamma^{00}$ is a quotient of
  $\Kk^\times/(\Kk^\times)^{00}$ by
  \cite[Lemmas~2.1--2.2]{jy-abelian}, and
  $\Kk^\times/(\Kk^\times)^{00}$ is profinite by
  \cite[Theorem~4.2]{jy-abelian}.  Then $\Gamma/\Gamma^{00}$ is
  profinite by Fact~\ref{kfact}(\ref{kf2}).
\end{proof}
Since we are working in the language with restricted analytic
functions, we have exponential and logarithm maps, and we can use
these to move between the multiplicative group and the additive group.
\begin{theorem} \label{mult-g00}
  Let $G$ be a definable subgroup of $\Kk^\times$.  Then $G^0 =
  G^{00}$.
\end{theorem}
\begin{proof}
  Using the short exact sequence
  \begin{equation*}
    1 \to \Oo^\times \to \Kk^\times \to \Gamma \to 1,
  \end{equation*}
  we can get a short exact sequence
  \begin{equation*}
    1 \to H \to G \to \Delta \to 1
  \end{equation*}
  where $H$ is a definable subgroup of $\Oo^\times$, namely $G \cap
  \Oo^\times$, and $\Delta$ is a definable subgroup of $\Gamma$,
  namely $\{v(x) : x \in G\}$.  By Corollary~\ref{gamma-subs-2},
  $\Delta^0 = \Delta^{00}$.  By Corollary~\ref{2-to-3}, it remains to
  show that $H^0 = H^{00}$.

  If $U = 1 + p^n \Oo$ is a small enough ball around 1, then the
  $p$-adic logarithm map gives an injective definable homomorphism
  \begin{equation*}
    \log_p : U \to \Qq_p.
  \end{equation*}
  The index of $U$ in $\Oo^\times$ is finite.  By
  Remark~\ref{finite-g00-rem}, we may replace $H$ with $H \cap U$, and
  assume that $H \subseteq U$.  Then $H \cong \log_p(H)$, and we are
  done by Theorem~\ref{additive-g00}.
\end{proof}
Next, we consider the case where $G$ is a subgroup of an irreducible
non-split torus.
\begin{remark} \label{tor-field-of-def}
  Let $T$ be an algebraic torus over $\Kk$.  Then $T$ is defined over
  $\Qp$.  Indeed, over \emph{any} perfect field $K$ of characteristic
  zero, $n$-dimensional algebraic tori are classified by actions of
  $\Gal(K)$ on $\Zz^n$ \cite[Theorem~2.1]{P-R-AG-book}.  Boundedness
  of $\Gal(\Qp)$ implies that $\Gal(\Kk) \cong \Gal(\Qp)$, and so the
  classification of algebraic tori is the same over both fields.  More
  precisely, the base-change functor from tori over $\Qp$ to tori over
  $\Kk$ is an equivalence of categories.  In particular, the functor
  is essentially surjective, as claimed.
\end{remark}
\begin{remark} \label{tor-compact}
  Let $T$ be an irreducible non-split torus over $\Qp$.  Then $T(\Qp)$
  is compact, by \cite[Theorem~3.1]{P-R-AG-book}.
\end{remark}
\begin{theorem}\label{torus-g00}
  Let $T$ be an irreducible non-split torus over $\Kk$.  Let $G$ be a
  definable subgroup of $T(\Kk)$.  Then $G^0 = G^{00}$.
\end{theorem}
\begin{proof}
  By Remark~\ref{tor-field-of-def},
  $T$ is definable over $\Qp$.
    Let $n = \dim(T)$.
  By properties of $p$-adic Lie groups, there is a neighborhood $U$ of
  $1$ in $T(\Qp)$ such that $U$ is a subgroup and $U$ is isomorphic to
  a ball in $\Qp^n$ via an analytic logarithm map.  For example,
  $T(\Qp)$ and $\Qp^n$ have isomorphic Lie algebras; apply
  \cite[Corollary~18.18]{schneider}.

  Therefore, there is a $\Qp$-definable open subgroup $U \subseteq
  T(\Kk)$ and a definable injective homomorphism $\log : U \to \Kk^n$.
  Note that $U(\Qp)$ has finite index in $T(\Qp)$ because $T(\Qp)$ is
  compact (by Remark~\ref{tor-compact}) and $U(\Qp)$ is open.
  Then finitely many
  translates of $U(\Qp)$ cover $T(\Qp)$.  As $\Kk \succ \Qp$,
  finitely many translates of $U$ cover $T(\Kk)$.  Thus $U$ has finite
  index in $T(\Kk)$.  By Remark~\ref{finite-g00-rem}, we may replace
  $G$ with the finite index subgroup $G \cap U$, and assume that $G
  \subseteq U$.  Then $G$ is definably isomorphic to a subgroup of
  $\Kk^n$ via the logarithm map, and so $G^0 = G^{00}$ by Theorem~\ref{additive-g00}.
\end{proof}
So we have seen that $G^0 = G^{00}$ when $G$ is a definable subgroup
of the additive group, the multiplicative group, or an irreducible
non-split torus.
\begin{fact} \label{decompose}
  If $V$ is a connected abelian linear algebraic group over a field
  $K$ of characteristic zero, then there is a chain of algebraic
  subgroups (over $K$):
  \begin{equation*}
    1 = V_0 \subseteq V_1 \subseteq \cdots \subseteq V_n = V
  \end{equation*}
  such that each quotient $V_i/V_{i-1}$ is one of the following
  algebraic groups:
  \begin{enumerate}
  \item The additive group $\Gg_a$.
  \item The multiplicative group $\Gg_m$.
  \item An irreducible non-split torus.
  \end{enumerate}
  This follows from the fact that $V$ is a direct product of a vector
  group and a torus, and a torus decomposes into irreducible tori
  which are either split ($\Gg_m$) or non-split.  See
  \cite[Corollary~16.15]{Milne}.
\end{fact}

By combining Theorems~\ref{additive-g00}, \ref{mult-g00}, and
\ref{torus-g00}, we get the following:
\begin{theorem} \label{qpan-aff}
  Let $G$ be an affine definable group (in $\Kk \succ \Qq_{p,\an}$).
  Then $G^0 = G^{00}$.
\end{theorem}
\begin{proof}
  By Corollary~\ref{to-abelian}, we may assume that $G$ is abelian.
  Let $V$ be the linear algebraic group such that $G \subseteq
  V(\Kk)$.  Replacing $V$ with the Zariski closure of $G$, we may
  assume that $G$ is Zariski dense in $V$.  Then $V$ is abelian.  By
  Remark~\ref{finite-g00-rem}, we may replace $G$ with a finite index
  subgroup.  Therefore, we may replace $V$ with its connected
  component, and assume that $V$ is connected.

  Let $\{V_i\}_{0 \le i \le n}$ be as in Fact~\ref{decompose}.  For
  each $i$, consider the map
\begin{equation*}
  G \cap V_i(K) \hookrightarrow V_i(K) \to (V_i/V_{i-1})(K).
\end{equation*}
Let $G_i$ be the image.  The kernel is $G \cap V_{i-1}(K)$.  Then we
have an ascending chain of definable subgroups
\begin{equation*}
  1 = G \cap V_0(K) \subseteq G \cap V_1(K) \subseteq \cdots \subseteq G \cap V_n(K) = G
\end{equation*}
such that the consecutive quotients are the definable groups $G_i
\subseteq (V_i/V_{i-1})(K)$.  By Corollary~\ref{2-to-3}, it suffices
to show that $(G_i)^0 = (G_i)^{00}$ for each $i$.  Therefore, we
reduce to the case where $V$ is one of the following:
\begin{enumerate}
\item The additive group.
\item The multiplicative group.
\item An irreducible non-split torus.
\end{enumerate}
These cases are handled by Theorems~\ref{additive-g00},
\ref{mult-g00}, and \ref{torus-g00}, respectively.
\end{proof}

\begin{theorem} \label{main-thm-2}
  Let $G$ be a definable group in a highly saturated elementary extension of $\Qp$.
  Suppose one of the following holds:
  \begin{enumerate}
  \item $G$ is affine.
  \item $G$ is defined over $\Qp$.
  \end{enumerate}
  Then $G^0 = G^{00}$.
\end{theorem}
\begin{proof} ~
  \begin{enumerate}
  \item Theorem~\ref{qpan-aff}, plus Corollary~\ref{reducts} to change the language.
  \item Apply Lemma~\ref{abelian-affine} to get a $\Qp$-definable short exact sequence
    \begin{equation*}
      1 \to A \to G \to H \to 1
    \end{equation*}
    where $A$ is abelian-by-finite and $H$ is affine.  Then $H^0 =
    H^{00}$ by part (1), and $A^0 = A^{00}$ by the abelian case
    \cite[Theorem~4.2]{jy-abelian} plus Remark~\ref{finite-g00-rem}.
    By Corollary~\ref{2-to-3}, $G^0 = G^{00}$.  \qedhere
  \end{enumerate}
\end{proof}
\begin{remark} \label{finite-extension}
  In this section we have been working with $\Qp$ rather than a
  finite extension $K/\Qp$.  Nevertheless, Theorems~\ref{qpan-aff} and
  \ref{main-thm-2} generalize to this setting, essentially because
  $K_{\an}$ is interpretable in $\Qq_{p,\an}$ via a $\Qp$-linear map $K
  \cong \Qp^d$, for $d = [K : \Qp]$.  Under this interpretation,
  $GL_n(K)$ is interpreted as a subgroup of $GL_{nd}(\Qp)$, and
  therefore any affine group in (an elementary extension of) $K$ is
  interpreted as an affine group in (an elementary extension of)
  $\Qp$.  This shows that Theorem~\ref{qpan-aff} extends from $\Qp$ to
  its finite extensions, and then the other proofs carry through with
  minimal changes.
\end{remark}
Theorem~\ref{main-thm-2} has the following corollary:
\begin{corollary}
  \label{final-corollary}
  Let $G$ be a definably amenable group defined in $\Qp$.  There is a
  finite index definable subgroup $E \subseteq G$ and a finite normal
  subgroup $F \lhd E$ such that the quotient $E/F$ is isomorphic to an
  open subgroup of an algebraic group over $\Qp$.
\end{corollary}
\begin{proof}
  Like Corollary~4.3 in \cite{jy-abelian}.
\end{proof}

\begin{acknowledgment}
    The first author was supported by the National Natural Science
    Foundation of China (Grant No.\@ 12101131) and the Ministry of
    Education of China (Grant No.\@ 22JJD110002). The second author was supported by the National Social Fund of China (Grant No.\@ 20CZX050).  We would like to thank an anonymous referee who caught several errors.
\end{acknowledgment}

\bibliographystyle{plain} \bibliography{bibliography}{}

\end{document}